\documentclass[12pt]{amsart}
\usepackage{amsmath,amssymb,amsbsy,amsfonts,amsthm,latexsym,amsopn,amstext,
                                               amsxtra,euscript,amscd}
\begin{document}

\newfont{\teneufm}{eufm10}
\newfont{\seveneufm}{eufm7}
\newfont{\fiveeufm}{eufm5}
%
\newfam\eufmfam
                \textfont\eufmfam=\teneufm \scriptfont\eufmfam=\seveneufm
                \scriptscriptfont\eufmfam=\fiveeufm
%
%
\def\frak#1{{\fam\eufmfam\relax#1}}
%


\def\bbbr{{\rm I\!R}} 
\def\bbbm{{\rm I\!M}}
\def\bbbn{{\rm I\!N}} 
\def\bbbf{{\rm I\!F}}
\def\bbbh{{\rm I\!H}}
\def\bbbk{{\rm I\!K}}
\def\bbbp{{\rm I\!P}}
\def\bbbone{{\mathchoice {\rm 1\mskip-4mu l} {\rm 1\mskip-4mu l}
{\rm 1\mskip-4.5mu l} {\rm 1\mskip-5mu l}}}
\def\bbbc{{\mathchoice {\setbox0=\hbox{$\displaystyle\rm C$}\hbox{\hbox
to0pt{\kern0.4\wd0\vrule height0.9\ht0\hss}\box0}}
{\setbox0=\hbox{$\textstyle\rm C$}\hbox{\hbox
to0pt{\kern0.4\wd0\vrule height0.9\ht0\hss}\box0}}
{\setbox0=\hbox{$\scriptstyle\rm C$}\hbox{\hbox
to0pt{\kern0.4\wd0\vrule height0.9\ht0\hss}\box0}}
{\setbox0=\hbox{$\scriptscriptstyle\rm C$}\hbox{\hbox
to0pt{\kern0.4\wd0\vrule height0.9\ht0\hss}\box0}}}}
\def\bbbq{{\mathchoice {\setbox0=\hbox{$\displaystyle\rm
Q$}\hbox{\raise 0.15\ht0\hbox to0pt{\kern0.4\wd0\vrule
height0.8\ht0\hss}\box0}} {\setbox0=\hbox{$\textstyle\rm
Q$}\hbox{\raise 0.15\ht0\hbox to0pt{\kern0.4\wd0\vrule
height0.8\ht0\hss}\box0}} {\setbox0=\hbox{$\scriptstyle\rm
Q$}\hbox{\raise 0.15\ht0\hbox to0pt{\kern0.4\wd0\vrule
height0.7\ht0\hss}\box0}} {\setbox0=\hbox{$\scriptscriptstyle\rm
Q$}\hbox{\raise 0.15\ht0\hbox to0pt{\kern0.4\wd0\vrule
height0.7\ht0\hss}\box0}}}}
\def\bbbt{{\mathchoice {\setbox0=\hbox{$\displaystyle\rm
T$}\hbox{\hbox to0pt{\kern0.3\wd0\vrule height0.9\ht0\hss}\box0}}
{\setbox0=\hbox{$\textstyle\rm T$}\hbox{\hbox
to0pt{\kern0.3\wd0\vrule height0.9\ht0\hss}\box0}}
{\setbox0=\hbox{$\scriptstyle\rm T$}\hbox{\hbox
to0pt{\kern0.3\wd0\vrule height0.9\ht0\hss}\box0}}
{\setbox0=\hbox{$\scriptscriptstyle\rm T$}\hbox{\hbox
to0pt{\kern0.3\wd0\vrule height0.9\ht0\hss}\box0}}}}
\def\bbbs{{\mathchoice
{\setbox0=\hbox{$\displaystyle     \rm S$}\hbox{\raise0.5\ht0\hbox
to0pt{\kern0.35\wd0\vrule height0.45\ht0\hss}\hbox
to0pt{\kern0.55\wd0\vrule height0.5\ht0\hss}\box0}}
{\setbox0=\hbox{$\textstyle        \rm S$}\hbox{\raise0.5\ht0\hbox
to0pt{\kern0.35\wd0\vrule height0.45\ht0\hss}\hbox
to0pt{\kern0.55\wd0\vrule height0.5\ht0\hss}\box0}}
{\setbox0=\hbox{$\scriptstyle      \rm S$}\hbox{\raise0.5\ht0\hbox
to0pt{\kern0.35\wd0\vrule height0.45\ht0\hss}\raise0.05\ht0\hbox
to0pt{\kern0.5\wd0\vrule height0.45\ht0\hss}\box0}}
{\setbox0=\hbox{$\scriptscriptstyle\rm S$}\hbox{\raise0.5\ht0\hbox
to0pt{\kern0.4\wd0\vrule height0.45\ht0\hss}\raise0.05\ht0\hbox
to0pt{\kern0.55\wd0\vrule height0.45\ht0\hss}\box0}}}}
\def\bbbz{{\mathchoice {\hbox{$\sf\textstyle Z\kern-0.4em Z$}}
{\hbox{$\sf\textstyle Z\kern-0.4em Z$}} {\hbox{$\sf\scriptstyle
Z\kern-0.3em Z$}} {\hbox{$\sf\scriptscriptstyle Z\kern-0.2em
Z$}}}}
\def\ts{\thinspace}

\newtheorem{theorem}{Theorem}
\newtheorem{lemma}[theorem]{Lemma}
\newtheorem{claim}[theorem]{Claim}
\newtheorem{cor}[theorem]{Corollary}
\newtheorem{prop}[theorem]{Proposition}
\newtheorem{definition}{Definition}
\newtheorem{question}[theorem]{Open Question}

\def\squareforqed{\hbox{\rlap{$\sqcap$}$\sqcup$}}
\def\qed{\ifmmode\squareforqed\else{\unskip\nobreak\hfil
\penalty50\hskip1em\null\nobreak\hfil\squareforqed
\parfillskip=0pt\finalhyphendemerits=0\endgraf}\fi}

\def\cA{{\mathcal A}}
\def\cB{{\mathcal B}}
\def\cC{{\mathcal C}}
\def\cD{{\mathcal D}}
\def\cE{{\mathcal E}}
\def\cF{{\mathcal F}}
\def\cG{{\mathcal G}}
\def\cH{{\mathcal H}}
\def\cI{{\mathcal I}}
\def\cJ{{\mathcal J}}
\def\cK{{\mathcal K}}
\def\cL{{\mathcal L}}
\def\cM{{\mathcal M}}
\def\cN{{\mathcal N}}
\def\cO{{\mathcal O}}
\def\cP{{\mathcal P}}
\def\cQ{{\mathcal Q}}
\def\cR{{\mathcal R}}
\def\cS{{\mathcal S}}
\def\cT{{\mathcal T}}
\def\cU{{\mathcal U}}
\def\cV{{\mathcal V}}
\def\cW{{\mathcal W}}
\def\cX{{\mathcal X}}
\def\cY{{\mathcal Y}}
\def\cZ{{\mathcal Z}}

\def\fF{\mathfrak F}

\newcommand{\comm}[1]{\marginpar{%
\vskip-\baselineskip 
\raggedright\footnotesize
\itshape\hrule\smallskip#1\par\smallskip\hrule}}




\newcommand{\ignore}[1]{}

\def\vec#1{\mathbf{#1}}

\def\e{\mathbf{e}}



\def\GL{\mathrm{GL}}

\hyphenation{re-pub-lished}

\def\rank{{\mathrm{rk}\,}}
\def\dd{{\mathrm{dyndeg}\,}}

\def\A{\mathbb{A}}
\def\B{\mathbf{B}}
\def \C{\mathbb{C}}
\def \E{\mathbf{E}}
\def \F{\mathbb{F}}
\def \K{\mathbb{K}}
\def \Z{\mathbb{Z}}
\def \P{\mathbb{P}}
\def \R{\mathbb{R}}
\def \Q{\mathbb{Q}}
\def \N{\mathbb{N}}
\def \Z{\mathbb{Z}}

\def \nd{{\, | \hspace{-1.5 mm}/\,}}

\def\vh{\mathbf{h}}
\def\e{\mathbf{e}}
\def\eM{{\e}_M}
\def\ep{{\e}_p}
\def\eqk{{\e}_{m_k}}

\def\Zn{\Z_n}

\def\Fp{\F_p}
\def \fp{\Fp^*}
\def\\{\cr}
\def\({\left(}
\def\){\right)}
\def\fl#1{\left\lfloor#1\right\rfloor}
\def\rf#1{\left\lceil#1\right\rceil}
\def\mand{\qquad\mbox{and}\qquad}

\setlength{\textheight}{43pc}
\setlength{\textwidth}{28pc}

\title[Iterations of Multivariate Polynomials]
{Pseudorandom Numbers and Hash Functions
from Iterations of Multivariate Polynomials}

\author{Alina~Ostafe}
\address{Institut f\"ur Mathematik, Universit\"at Z\"urich,
Winterthurerstrasse 190 CH-8057, Z\"urich, Switzerland}
\email{alina.ostafe@math.uzh.ch}

\author{Igor E.~Shparlinski} 
\address{Department of Computing, Macquarie University\\ NSW 2109, Australia}
\email{igor@ics.mq.edu.au}

\begin{abstract}  Dynamical systems generated by iterations of multivariate polynomials with slow degree growth have 
proved to admit good  estimates of exponential sums along their orbits which in turn 
lead to rather stronger bounds on the discrepancy for pseudorandom 
vectors generated by these iterations. 
Here we add new arguments to our original approach and 
also extend some of our recent constructions and results  
to more general orbits of polynomial iterations which 
may involve distinct polynomials as well. 
Using this  construction we design a new class of hash functions from iterations of polynomials and use our estimates to motivate their ``mixing'' properties. 
\end{abstract}

\maketitle


\section{Introduction}

\subsection{Background}

For a system of $m+1$ polynomials $\cF = \{f_0, \ldots ,f_{m}\}$ 
in $m+1$ variables over a
ring $\cR$ one can naturally define a  dynamical system
generated by its iterations:
$$
f_i^{(0)}=f_i, \qquad f_i^{(k)}= f_i^{(k-1)}(f_0, \ldots ,f_{m}), \qquad k=1,2 \ldots\ ,  
$$
for each $i = 0, \ldots, m$, 
see~\cite{EvWa,FomZel,Jon,Shp,Silv1,Silv2}
and references therein 
for various aspects of such dynamical systems. 
In particular, the length and the distribution 
of elements in the orbits of such 
dynamical systems, starting from an initial value 
$\(u_{0,0}, \ldots, u_{0,m}\)\in \cR^{m+1}$,
 have been of primal interest.

In the special case of one linear univariate polynomial over 
a residue ring or a finite field such iterations are 
known as {\it linear congruential generators\/}, 
which  have been successully 
used for decades in   Quasi-Monte Carlo methods, see~\cite{Nied1,Nied2}.
On the other hand, in cryptographic settings, such linear generators 
have been the subject of various attacks~\cite{ContShp,FHKLS,JS,Kraw,Lag} 
and thus   are not recommended for  cryptographic purposes. 
It should be noted that nonlinear generators have also been attacked~\cite{BGGS1,BGGS2,GoGuIb,GuIb},
but the attacks are much weaker and do not rule out their 
use for cryptographic purposes (provided reasonable precautions are 
made). 
Although linear congruential generators have been used quite sucessfully
for  Quasi-Monte Carlo methods, their linear structure shows in 
these applications too and often limits their applicability,
see~\cite{Nied1,Nied2}. 

 Motivated by these potential applications, 
the statistical uniformity of the distribution 
(measured by the discrepancy) of one and  multidimensional 
nonlinear polynomial generators 
have been intensively 
studied in~\cite{GNS,GG,NiSh1,NiSh2,NiWi,TopWin}. However, all
previously known results are nontrivial only for those 
polynomial generators
that produce sequences of extremely large period, which could
be hard to achieve in practice. The reason
behind this is that  the degree of iterated polynomial
systems grows exponentially, and that in all previous results on the general case 
the saving over the trivial bound has been logarithmic.
Moreover, it is easy to see that in the one dimensional
case (that is, for $m=0$) the exponential growth of 
the degree of iterations of a nonlinear polynomial is
unavoidable. One also expects the same behaviour 
in the mulitidimensional case for ``random'' polynomials
$f_0, \ldots ,f_{m}$. However, as it has been 
shown in~\cite{OstShp} for some specially selected
polynomials $f_0, \ldots ,f_{m}$ the degree may grow 
significantly slower, a result that leads to 
 much better estimates
of exponential sums, and thus of discrepancy, for vectors 
generated by these iterations.

Furthermore, it is shown in~\cite{Ost},
that  in the case when such a 
polynomial map generates 
a permutation of the corresponding vector space, 
one can get better results ``on average" over all initial values.
It is also noticed in~\cite{Ost} that in fact one
can avoid the use of the {\it Weil
bound\/} (see~\cite[Chapter~5]{LN}) of exponential sums 
and achieve a better result with a more elementary argument.

\subsection{Our results}

Here, as in~\cite{Ost}, we continue to study 
 the polynomial
systems of~\cite{OstShp} and exploit  the linearity with respect to 
one variable and polynomial degree growth with respect to 
the other variables. This leads to a direct improvement of the 
results of~\cite{OstShp}. 
This new approach also 
allows us to consider a slightly more general polynomial dynamical 
systems, where at each step a different polynomial
map can be used, thus extending  
those of~\cite{OstShp}. 
The argument is based 
on an elementary identity for exponential sums with linear polynomials
and also on counting zeros of multivariate polynomials in 
finite fields.
  
We remark that since 
the Weil bound is not needed anymore, one can certainly obtain 
analogues of our results for residue rings (although 
counting the number of solutions of multivariate polynomial
congruences may require more efforts than in the finite 
field settings). 

Furthermore, in~\cite{Ost,OstShp}
only the truncated vectors (consisting of $m$ components
of the total output $(m+1)$-dimensional vectors) are investigated.
Here we show that in fact the whole output vectors
can be studied, however for this we require a 
very deep result of Bourgain,  Glibichuk and 
Konyagin~\cite{BGK} (for generalisation to 
residue rings one can also use the results 
of~\cite{Bourg1,Bourg2}).

Finally, we propose a construction of a hash function from 
polynomial maps. Although we make no claims of security or 
efficiency, we note that our results show that this
hash function has ``random-like'' behaviour. 

Hash functions from walks on the set of isogenous elliptic curves
generated by low degree isogenies, and 
their cryptographic applications, are considered in~\cite{CGL,JMV}.
Alternatively these walks can be described as sequences of rational
function transformations on the coefficients of 
Weierstrass equations on elliptic curves, see~\cite{Silv0}
for a background.  
We hope that our results maybe useful for studying further properties of
such walks, for example, in showing that the hash function of~\cite{CGL,JMV}
has sufficiently  uniformly distributed 
outputs and maybe used as a secure pseudorandom number generator.  

\section{Construction} 

\subsection{Polynomial systems}

Let  $\F$ be an arbitrary field. 
As in~\cite{OstShp}, we consider a system $\cF = \{F_0, \ldots ,F_m\}$  
of $m+1$ polynomials in $\F[X_0, \ldots ,X_m]$ satisfying the following
conditions
\begin{equation}
\label{eq:Polys}
\begin{split}
F_0(X_0, \ldots ,X_m)& = X_0G_0(X_1,\ldots,X_m)+H_0(X_1,\ldots,X_m),\\
F_1(X_0, \ldots ,X_m)&=X_1G_1(X_2,\ldots,X_m)+H_1(X_2,\ldots,X_m),\\
  &\ldots  \\
F_{m-1}(X_0, \ldots ,X_m)&=X_{m-1}G_{m-1}(X_m)+H_{m-1}(X_m),\\
F_m(X_0, \ldots ,X_m) &= g_mX_m+h_m,
\end{split}
\end{equation}
where 
$$
g_m,h_m\in\F,\quad g_m \ne 0, \quad G_i,H_i\in\F[X_{i+1},\ldots,X_m],\quad i=0,\ldots,m-1.
$$ We also impose the condition that each polynomial $G_i$, $i =0, \ldots, m-1$,  has the {\it unique leading monomial\/} $X_{i+1}^{s_{i,i+1}}\ldots X_m^{s_{i,m}}$, that is,
\begin{equation}
\label{eq:Cond1}
G_i(X_{i+1},\ldots,X_m) = g_iX_{i+1}^{s_{i,i+1}}\ldots X_m^{s_{i,m}} + 
\widetilde{G_i}(X_{i+1},\ldots,X_m),  
\end{equation}
where 
\begin{equation}
\label{eq:Cond2}
g_i \in \F^*,  \qquad \deg_{X_j} \widetilde{G_i} <  s_{i,j}, 
 \qquad \deg_{X_j} H_i\le  s_{i,j},
\end{equation}
for $i=0,\ldots,m-1$, $j = i+1, \ldots, m$.

Given an integral upper triangular matrix 
\begin{equation}
\label{eq:Mat S}
S=   \(\begin{matrix}        1 & s_{0,1} &  s_{0,2} & \ldots & s_{0,m} \\
      0 & 1 & s_{1,2} & \ldots & s_{1,m} \\
      &&\ldots &&\\
      0 & 0 & 0 & \ldots & 1
   \end{matrix}\)
\end{equation}
define $\fF(S,m)$ the set of all such polynomial systems 
of the form~\eqref{eq:Polys}
satisfying the conditions~\eqref{eq:Cond1} 
and~\eqref{eq:Cond2}.

For an integer  $m\ge 1$ and an integral  matrix 
$S$  of the form~\eqref{eq:Mat S},  we consider a sequence of, 
not necessarily distinct, polynomial systems 
\begin{equation}
\label{eq:Poly Syst}
\cF_k=\{F_{k,0}, \ldots, F_{k,m}\} \in \fF(S,m),
\qquad  k=1, 2,\ldots\,.
\end{equation}
We consider the sequence of polynomials $ F_i^{(j)}$ 
defined by the recurrence relation
\begin{equation}
\label{eq:Poly Rec}
F_i^{(0)}= X_i,  
\qquad F_i^{(k)}= F_{k,i}(F_0^{(k-1)}, \ldots ,F_m^{(k-1)}), 
\qquad k=1, 2,\ldots\,.
\end{equation}
In particular, $\cF_0$ denotes the identity map. 

As in~\cite[Lemma~1]{Ost}, we have the following characterization of the polynomials $F_i^{(k)}$, which in turn generalises and 
refines~\cite[Lemma~1]{OstShp}. We note that 
unfortunately in~\cite{OstShp}
the unique leading monomial condition~\eqref{eq:Cond2} is given in the 
form  $\deg  \widetilde{G_i} <  \deg G_i$ instead of the
required $\deg_{X_j} \widetilde{G_i} <  \deg_{X_j} G_i$,
$0 \le i < j \le m$, 
that is actually used in the proof of~\cite[Lemma~1]{OstShp}. 

\begin{lemma}
\label{lem:LinTerm+Deg} Let $\cF_k \in \fF(S,m)$ be a sequence of 
 polynomial  systems~\eqref{eq:Poly Syst}. 
Then for the polynomials $F_i^{(k)}$ given by~\eqref{eq:Poly Rec} we have
$$
F_i^{(k)} = X_i \widetilde{G}_{k,i}  (X_{i+1},\ldots,X_m) + 
\widetilde{H}_{k,i} (X_{i+1},\ldots,X_m), 
$$
where $\widetilde{G}_{k,i}, \widetilde{H}_{k,i}  \in \F[X_{i+1},\ldots,X_m]$ and
\begin{eqnarray*}
\deg \widetilde{G}_{k,i} &=&\frac{1}{(m-i)!}k^{m-i}s_{i,i+1}\ldots s_{m-1,m}+\psi_i(k),\quad  0 \le i \le m-1,\\
\deg \widetilde{G}_{k,m}&=& 0,
\end{eqnarray*}
with some  polynomials $\psi_i(T) \in \Q[T]$  of degree  $\deg \psi_i <m-i$.
\end{lemma} 

\begin{proof}
Writing $F_{k,i} = X_i G_{k,i} + H_{k,i}$
we get
$$
F_i^{(k)}  = F_i^{(k-1)} G_{k,i}\(F_{i+1}^{(k-1)}, \ldots, F_{m}^{(k-1)}\)
+H_{k,i}\(F_{i+1}^{(k-1)}, \ldots, F_{m}^{(k-1)}\).
$$
Thus an easy inductive argument implies that 
$$
F_i^{(k)} = X_i \widetilde{G}_{k,i}  (X_{i+1},\ldots,X_m) + 
\widetilde{H}_{k,i} (X_{i+1},\ldots,X_m)
$$
for some polynomials $\widetilde{G}_{k,i}, \widetilde{H}_{k,i}  \in \F[X_{i+1},\ldots,X_m]$, 
where  $i =0, \ldots, m$,  $k=1, 2,\ldots$. 

For the asymptotic formulas for the degrees of the polynomials 
$\widetilde{G}_{k,i}$ 
see~\cite[Lemma~1]{OstShp} where it is given 
for  $\deg F_i^{(k)}$. 
We note that 
in~\cite{OstShp} only the case when at each step 
the same polynomial system $\cF_k = \cF$ is applied but the 
proof holds for distinct systems $\cF_k \in \fF(S,m)$ 
without any changes.  
Indeed, 
let 
$$
d_{k,i} = \deg (X_i \widetilde{G}_{k,i}) = 1 + \deg \widetilde{G}_{k,i}, 
\qquad i =0, \ldots, m,\ k =1,2 \ldots.
$$
Then the result follows 
immediately 
from the recursive formula
$$
(d_{k,0}, \ldots, d_{k,m})^t = S^{k}  (1, \ldots, 1) 
$$
implied by~\eqref{eq:Cond1} and \eqref{eq:Cond2},
where 
$$S=   \(\begin{matrix}        1 & s_{0,1} &  s_{0,2} & \ldots & s_{0,m} \\
      0 & 1 & s_{1,2} & \ldots & s_{1,m} \\
      &&\ldots &&\\
      0 & 0 & 0 & \ldots & 1
   \end{matrix}\),
 $$
and $\vec{d}^T$ means the transposition 
of the vector $\vec{d}$, 
see the proof of~\cite[Lemma~1]{OstShp} for more details.
\end{proof}

\subsection{Vector sequences}

Given a sequence of polynomial systems~\eqref{eq:Poly Syst}, 
we fix a vector $\vec{v} \in \F_p^{m+1}$ and
consider the sequence defined by a recurrence congruence
modulo a prime $p$ of the form
\begin{equation}
\label{eq:Gen}
u_{n+1,i}\equiv F_{n+1,i}(u_{n,0},\ldots,u_{n,m})\!\!\! \pmod p, \qquad n = 0,1,\ldots,
\end{equation}
with some {\it initial values\/}
$$(u_{0,0},\ldots,u_{0,m}) = \vec{v}.
$$ 
We also assume that $0 \le u_{n,i} < p$, $i=0,\ldots,m$, $n=0, 1, \ldots$.

Using the following vector notation
$$
\vec{w}_n=(u_{n,0},\ldots,u_{n,m})
$$
we have the recurrence relation
$$
\vec{w}_{n}=\cF_{n}(\vec{w}_{n-1}),   \qquad n =1,2, \ldots\,.
$$
In particular, for any $n,k\ge 0$ and $i=0,\ldots,m$ we have
$$
u_{n+k,i}=F_i^{(k)}(u_{n,0},\ldots,u_{n,m}), 
$$ 
where the polynomials $F_i^{(k)}$, 
$i=0,\ldots,m$, $k =1,2, \ldots$, are 
given by~\eqref{eq:Poly Rec}. 
Clearly the sequence of vectors $\vec{w}_{n}$ is eventually periodic
with some period $\tau \le p^{m+1}$.
We always  assume that the sequence is purely 
periodic, that is, 
$$
\vec{w}_{n+\tau} = \vec{w}_{n}, \qquad n =0,1, \ldots\,.
$$
As in~\cite{Ost,OstShp}, we sometimes discard the last component 
and define the truncated vectors 
$$
\vec{u}_n=(u_{n,0},\ldots, u_{n,m-1})
$$
However, here we introduce a new argument which allows us sometimes
to study full vectors $\vec{w}_n$.

\section{Exponential Sums and Discrepancy}
\label{sec:Distr}

\subsection{Preliminaries}

Assume that  the sequence $\{\vec{u}_n\}$ generated by~\eqref{eq:Gen} is 
purely periodic
with an arbitrary period $\tau$.
For integer vectors $\vec{a} = (a_0, \ldots, a_{m-1}) \in \Z^{m}$
and $\vec{b} = (b_0, \ldots, b_{m}) \in \Z^{m+1}$ we introduce
the exponential sums
$$
S_{\vec{a}}(N) =  \sum_{n=0}^{N-1} \ep\(\sum_{i=0}^{m-1} a_iu_{n,i}\)
\quad \text{and}\quad 
T_{\vec{b}}(N) =  \sum_{n=0}^{N-1} \ep\(\sum_{i=0}^{m} b_iu_{n,i}\), 
$$
where
$$
\ep(z) = \exp(2 \pi i z/p).
$$
Clearly, if $\vec{b} = (a_0, \ldots, a_{m-1}, 0)$ then we
simply have $S_{\vec{a}}(N) = T_{\vec{b}}(N)$, thus the sums
$T_{\vec{b}}(N)$ are direct generalisations of 
the sums  $S_{\vec{a}}(N)$ that have been treated 
in~\cite{Ost,OstShp}. Here we show that together 
with some additional arguments, one can obtain similar results
for the sums $T_{\vec{b}}(N)$.

Bounds of these sums can be used to estimate the discrepancy 
of the corresponding sequences, which is a widely accepted 
quantitative measure 
of uniformity of distribution of sequences, and thus good pseudorandom
sequences should (after an appropriate scaling) have a small discrepancy,
see~\cite{Nied1,Nied2}.

Given a sequence $\Gamma$ of $N$ points 
\begin{equation}
\label{eq:GenSequence}
\Gamma = \left\{(\gamma_{n,1}, \ldots, \gamma_{n,s})_{n=0}^{N-1}\right\}
\end{equation}
in the $s$-dimensional unit cube $[0,1)^s$
it is natural to measure the level of its statistical uniformity 
in terms of the {\it discrepancy\/} $\Delta(\Gamma)$. 
More precisely, 
$$
\Delta(\Gamma) = \sup_{B \subseteq [0,1)^s}
\left|\frac{T_\Gamma(B)} {N} - |B|\right|,
$$
where $T_\Gamma(B)$ is the number of points of  $\Gamma$
inside the box
$$
B = [\alpha_1, \beta_1) \times \ldots \times [\alpha_{s}, \beta_{s})
\subseteq [0,1)^s
$$
and the supremum is taken over all such boxes, see~\cite{DrTi,KuNi}.

Typically the bounds on the discrepancy of a 
sequence  are derived from bounds of exponential sums
with elements of this sequence. 
The relation is made explicit in 
 the celebrated {\it Erd\H os-Turan-Koksma
inequality\/}, see~\cite[Theorem~1.21]{DrTi},
which we  present in the following form.

\begin{lemma}
\label{lem:Kok-Szu} There exists a constant $C_s$ depending only 
on $s$ such that for any
integer $H > 1$ and any  sequence $\Gamma$ of $N$ points~\eqref{eq:GenSequence}
the discrepancy $\Delta(\Gamma)$
satisfies the following bound:
$$
\Delta(\Gamma) \le C_s\( \frac{1}{H}
+ \frac{1}{N}\sum_{ 0 < |\vec{h}| \le H}  
\prod_{j=1}^s \frac{1}{ |h_j| + 1}
\left| \sum_{n=0}^{N-1} \exp \( 2 \pi i\sum_{j=1}^{s}h_j\gamma_{n,j} \)
\right| \)
$$
where the sum is taken over all integers vectors
$\vec{h} = (h_1, \ldots, h_s) \in \Z^s$
with $|\vec{h}| = \max_{j = 1, \ldots, s} |h_j| < H$.
\end{lemma}

We always assume that  a finite field $\F_p$ of $p$ elements
is represented by  the set $\{0, 1, \ldots, p-1 \}$.
So for $u \in \F_p$ we always have $u/p \in [0,1)$ and
thus we can talk about the  discrepancy of vectors
over $\F_p$ after scaling them by $1/p$.

Throughout the paper, the implied constants in the 
symbols `$O$' and `$\ll$'
may occasionally, where obvious, depend on the matrix $S$
and the integer $m\ge 1$  (and are absolute otherwise). We recall that the notations $A = O(B )$ and $A\ll B$ are all equivalent to the assertion that the inequality $|A|\leq cB$ holds for some
constant $c > 0$.

\subsection{Arbitrary Systems}

Here we assume, exactly  as in~\cite{OstShp}, that 
all polynomial systems~\eqref{eq:Poly Syst} are the same,
that is $\cF_k = \cF$.  Our next results 
are a  direct   improvement 
of the estimate of~\cite[Theorem~4]{OstShp} for 
the sums $S_{\vec{a}}(N) $ 
and also an extension of such bound to more general sums $T_{\vec{b}}(N)$.

We need the following 
generalisation of the bound on exponential sums of~\cite[Lemma~2]{Ost}, 
which avoids using the Weil bound (see~\cite[Chapter~5]{LN})
and which is our main tool in improving the result of~\cite{OstShp}. 

\begin{lemma}
\label{lem:Elem}
Let $\cF \in \fF(S,m)$ with $s_{0,1} \ldots s_{m-1,m} \ne 0$,  then there is a positive integer $k_0$ depending only on  $S$ and $m$ such that for any integer vectors 
$$
\vec{k} = (k_1, \ldots, k_\nu), \quad \vec{l} = (l_1, \ldots, l_\nu), \quad 
\min \{k_1, \ldots, k_\nu, l_1, \ldots, l_\nu\} \ge k_0
$$
with  components that  are not permutations of each other
and  integer vector $\vec{a} = (a_0, \ldots, a_{m-1})$
with
$$
\gcd (a_0, \ldots, a_{m-1}, p) = 1,
$$
for the polynomial
$$
F_{\vec{a},\vec{k},\vec{l}} =   \sum_{i=0}^{m-1} \!a_i
\sum_{h=1}^\nu \(F_i^{(k_h)}-F_i^{(l_h)}\)
$$
where  the polynomials $F_i^{(k)}$ are given by~\eqref{eq:Poly Rec}, we have
$$
 \sum_{w_0, \ldots, w_m =1 }^p
\ep\( F_{\vec{a},\vec{k},\vec{l}}(w_{0},\ldots,w_{m})\)   \ll K^m p^{m}, 
$$
where 
$$
K = \max \{k_1, \ldots, k_\nu, l_1, \ldots, l_\nu\} .
$$
\end{lemma}

\begin{proof}  Let $s < m-1$ be the smallest integer such that $a_s\ne 0$.
By Lemma~\ref{lem:LinTerm+Deg} we have
\begin{equation*}
\begin{split}
F_{\vec{a}, \vec{k},\vec{l}}&(x_{0},\ldots,x_{m})  \\
 &  = \sum_{i=s}^{m-1} \!a_i x_i
\sum_{h=1}^\nu
\(\widetilde{G}_{k_h,i}(x_{i+1}, \ldots, x_m)-\widetilde{G}_{l_h,i}(x_{i+1}, \ldots, x_m)\)\\
& \qquad\qquad + \sum_{i=s}^{m-1} \!a_i  
\sum_{h=1}^\nu \(\widetilde{H}_{k_h,i}(x_{i+1}, \ldots, x_m) -\widetilde{H}_{l_h,i}(x_{i+1}, \ldots, x_m)\)\\
 &  = a_s x_s \sum_{h=1}^\nu\(\widetilde{G}_{k_h,s}(x_{s+1}, \ldots, x_m)-\widetilde{G}_{l_h,s}(x_{s+1}, \ldots, x_m)\) \\
& \qquad\qquad\qquad\qquad\qquad\qquad \qquad\qquad
+  \Psi_{\vec{a}, \vec{k},\vec{l}}(x_{s+1},\ldots,x_{m}) 
\end{split}
\end{equation*}
for a certain polynomial $ \Psi_{\vec{a}, \vec{k},\vec{l}}(x_{s+1},\ldots,x_{m}) \in \F_p[ x_{s+1},\ldots,x_{m}]$. 

Therefore, 
\begin{equation*}
\begin{split}
 &\sum_{x_0, \ldots, x_m =1 }^p   
\ep\( F_{\vec{a}, \vec{k},\vec{l}}(x_{0},\ldots,x_{m})\) \\
 & \quad   =p^{s} \sum_{x_{s+1}, \ldots, x_{m} =1 }^p
\ep\( \Psi_{\vec{a}, \vec{k},\vec{l}}(x_{s+1},\ldots,x_{m}) \)\\
& \qquad  \sum_{x_s=1}^p  \ep\( a_s x_s \sum_{h=1}^\nu\(\widetilde{G}_{k_h,s}(x_{s+1}, \ldots, x_m)-\widetilde{G}_{l_h,s}(x_{s+1}, \ldots, x_m)\)\).
\end{split}
\end{equation*}

Recalling the identity 
\begin{eqnarray}
\label{eq:lin sum}
  \sum_{u=1}^p \ep(c u) = \left\{\begin{array}{ll}
p,&\quad\text{if $c\equiv 0 \pmod p$,}\\
0,&\quad\text{if $c\not\equiv 0 \pmod p$,}
\end{array}
\right.
\end{eqnarray}
see~\cite[Equation~(5.9)]{LN1}, 
we conclude that  the sum over the variable $x_s$ is nonzero only 
if the polynomial 
$$
\Phi_{s,\vec{k},\vec{l}}= \sum_{h=1}^\nu(\widetilde{G}_{k_h,s}-\widetilde{G}_{l_h,s})
\in \F_p[X_{s+1},\ldots,X_m]
$$ 
is zero modulo $p$ at $(x_{s+1}, \ldots, x_m)$.  

Performing all trivial cancelations, without loss of generality we can also assume that the vectors $\vec{k}$ and $\vec{l}$ have no common elements. Thus, 
by Lemma~\ref{lem:LinTerm+Deg}, we see that if $\min \{k_1, \ldots, k_\nu, l_1, \ldots, l_\nu\} \ge k_0$
for a sufficiently large $k_0$ then  the polynomial  
$\Phi_{s,\vec{k},\vec{l}}$
is a nontrivial polynomial modulo $p$ of degree $O(K^{m-s}) = O(K^m)$. 
Also, a simple inductive argument shows that a modulo $p$ nontrivial polynomial 
in $r$ variables of degree $D$ may have only $O(Dp^{r-1})$
zeros modulo $p$, which concludes the proof. 
\end{proof}

\begin{theorem}
\label{thm:ExpSum S} 
Let  the sequence $\{\vec{u}_n\}$ be given by~\eqref{eq:Gen}
for     $\cF_k = \cF$, $k=1,2, \ldots$, with a 
polynomial 
system $\cF \in  \fF(S,m)$ of the form~\eqref{eq:Polys}
of total degree $d \ge 2$  
and such that $s_{0,1} \ldots s_{m-1,m} \ne 0$.
Assume that $\{\vec{w}_n\}$ is purely
periodic with period $\tau$.  Then for any fixed integer $\nu\ge 1$, 
 positive integer $N \le \tau$ and   nonzero vector $\vec{a}  \in \F_p^m$ 
the bound
$$
S_{\vec{a}}(N)    \ll N^{1-\beta_{m,\nu}}p^{\alpha_{m,\nu}}
$$
holds, where
$$
\alpha_{m,\nu}=\frac{m^2 + m\nu + m}{2\nu(m+\nu)} \qquad \text{and} 
\qquad \beta_{m,\nu} = \frac{1}{2\nu}
$$
and the implied constant depends  only on $d$, $m$
and $\nu$.
\end{theorem}

\begin{proof} We follow the same argument 
as in  the proof 
of~\cite[Theorem~4]{OstShp}, however 
instead of the Weil bound we use now Lemma~\ref{lem:Elem}
(and thus we optimise the parameters differently). 
  
In particular, as in~\cite{OstShp} we obtain  
that for any integer $K \ge k_0$,
\begin{equation}
\label{eq:S and W}
(K-k_0+1) |  S_{\vec{a}}(N)|  \le  W + K^2,
\end{equation}
where $k_0$ is the same as in Lemma~\ref{lem:Elem} and
$$
W = \left |\sum_{n=0}^{N-1}\sum_{k=k_0}^{K}
\e\(\sum_{i=0}^{m-1} a_i u_{n+k,i}\) \right|.
$$
Using the H{\"o}lder inequality we derive 
(again exactly the same way as in~\cite{OstShp}) 
$$
W^{2\nu}   \le  N^{2\nu-1}  
 \sum_{k_1,\ell_1, \ldots, k_\nu, \ell_\nu = k_0}^{K} 	\sum_{w_0, \ldots, w_m\in \F_{p}^{m+1}}
 	\e\!\(F_{\vec{a},\vec{k},\vec{l}}(w_0, \ldots, w_m)\).
$$
For  $O(K^\nu)$ vectors 
 $$
 (k_1 \ldots, k_\nu)  \qquad \text{and}\qquad  (\ell_1 \ldots, \ell_\nu)
 $$
which are permutations of each other,  we estimate  the inner sum 
trivially as $p^{m+1}$.

For the other $O(K^{2\nu})$ vectors, we apply 
Lemma~\ref{lem:Elem}  getting the
upper bound $K^mp^{m}$ for the inner sum. Hence, 
$$W^{2\nu} \le K^\nu N^{2\nu-1} p^{m+1} +  K^{m+2\nu} N^{2\nu-1} p^{m}.
$$
Inserting this bound in~\eqref{eq:S and W}, we derive
$$
S_{\vec{a}}(N) 
\ll K^{-1/2}N^{1-1/2\nu} p^{(m+1)/2\nu} +  K^{m/2\nu} N^{1-1/2\nu} p^{m/2\nu} +K.
$$
Choosing
$$K = \rf{p^{1/(m+\nu)}}$$
(and assuming that $p$ is large enough, so $K\ge k_0$),
after simple calculations  we obtain the desired result.
\end{proof}

Using Lemma~\ref{lem:Kok-Szu},
we  derive the following improvement
of~\cite[Theorem~6]{OstShp}.

\begin{cor}
\label{cor:Discr m-1} 
Let  the sequence $\{\vec{u}_n\}$ be given by~\eqref{eq:Gen}
for     $\cF_k = \cF$, $k=1,2, \ldots$, with a 
polynomial 
system $\cF \in  \fF(S,m)$ of the form~\eqref{eq:Polys}
of total degree $d \ge 2$  
and such that $s_{0,1} \ldots s_{m-1,m} \ne 0$.
Assume that $\{\vec{w}_n\}$ is purely
periodic with period $\tau$. Then for any fixed integer $\nu\ge 1$, 
and any positive integer $N \le \tau$, 
the discrepancy  of the sequence 
$$
\(\frac{u_{n,0}}{p}, \ldots, \frac{u_{n,m-1}}{p}\),
\qquad n = 0,\ldots, N-1,
$$
satisfies the bound 
$O\(p^{\alpha_{m,\nu}}N^{-\beta_{m,\nu}} (\log p)^{m}\)$, 
where
$$
\alpha_{m,\nu}=\frac{m^2 + m\nu + m }{2\nu(m+\nu)} \qquad \text{and} 
\qquad \beta_{m,\nu} = \frac{1}{2\nu}
$$
and the implied constant depends  only on $d$, $m$
and $\nu$.
\end{cor}

We note that the  values of $\alpha_{m,\nu}$ and 
$\beta_{m,\nu}$ in Theorem~\ref{thm:ExpSum R}
and Corollary~\ref{cor:Discr m-1} improve 
on the values
$$
\alpha_{m,\nu}=\frac{2m^2 + 2m\nu + 2m +\nu}{4\nu(m+\nu)} \qquad \text{and} 
\qquad \beta_{m,\nu} = \frac{1}{2\nu}
$$
from~\cite{OstShp}. 
In particular, both Theorem~\ref{thm:ExpSum S}
and Corollary~\ref{cor:Discr m-1} are nontrivial if 
$\tau \ge N \ge p^{m+\varepsilon}$ with fixed $\varepsilon > 0$
(while the corresponding bounds of~\cite{OstShp}
are nontrivial only if $\tau \ge N \ge p^{m+1/2 +\varepsilon}$).

\begin{theorem}
\label{thm:ExpSum R} 
Let  the sequence $\{\vec{u}_n\}$ be given by~\eqref{eq:Gen}
for     $\cF_k = \cF$, $k=1,2, \ldots$, with a 
polynomial 
system $\cF \in  \fF(S,m)$ of the form~\eqref{eq:Polys}
of total degree $d \ge 2$  
and such that $s_{0,1} \ldots s_{m-1,m} \ne 0$.
Assume that $\{\vec{w}_n\}$ is purely
periodic with period $\tau$.  Then for any fixed real $\varepsilon > 0$,
there exist $\delta > 0$ such that for 
for any positive integer  $N$ with 
$\tau \ge N \ge p^{m+\varepsilon}$ and   nonzero vector $\vec{b}  \in \F_p^{m+1}$ 
the bound
$$
 T_{\vec{b}}(N)  \ll Np^{-\delta} 
$$
holds
and the implied constant depends  only on $d$, $m$
and $\varepsilon$.
\end{theorem}

\begin{proof} If 
$\gcd(b_0, \ldots, b_{m-1} ,p) = 1$ then the same argument 
as in  the proof  Theorem~\ref{thm:ExpSum R} leads to a fully
analogous  bound
$$
T_{\vec{b}}(N)    \ll N^{1-\beta_{m,\nu}}p^{\alpha_{m,\nu}}.
$$
Thus for $\tau \ge N \ge p^{m+\varepsilon}$, taking a 
sufficiently large $\nu$ we obtain the desired estimate.

So it remains to consider 
the case 
$$
b_0 \equiv \ldots  \equiv b_{m-1}  \equiv 0 \pmod p \mand 
\gcd(b_{m} ,p) = 1, 
$$
in which case we simply obtain 
$$
T_{\vec{b}}(N) =  \sum_{n=0}^{N-1} \ep\(b_mu_{n,m}\). 
$$
A trivial inductive argument shows that 
\begin{equation}
\label{eq:um expl ne 1}
u_{n,m} = g_m^n u_{0,m} + \frac{g_m^n-1}{g_m-1} h_m, \qquad n =0, 1, \ldots, 
\end{equation}
if $g_m \ne 1$  and 
\begin{equation}
\label{eq:um expl = 1}
u_{n,m} = n h_m, \qquad n =0, 1, \ldots, 
\end{equation}
if $g_m= 1$ (where $g_m$ and $h_m$ are as in~\eqref{eq:Polys}).

We consider the case $g_m\ne 1$ first in which we obtain
$$
T_{\vec{b}}(N) =  \ep(-b_m h_m(g_m-1)^{-1})\sum_{n=0}^{N-1} 
\ep\(b_m g_m^n \(u_{0,m} + h_m(g_m-1)^{-1}\)\). 
$$
Clearly, 
if $t$ is the multiplicative order 
of $g_m$ then we see from~\eqref{eq:um expl ne 1} 
that $u_{n,m}$, $n =0, 1, \ldots,$ takes exactly $t$ 
distinct values. Since the truncated vector 
$\vec{u}_n$ takes at most $p^{m}$ values we see that the 
full vector $\vec{w}_n$ takes at most $tp^m$ values.
Thus 
$$\tau \le p^mt.
$$ 
Using the condition $\tau \ge N \ge p^{m+\varepsilon}$ we
obtain 
\begin{eqnarray}
\label{eq:t large}
t \ge p^{\varepsilon}.
\end{eqnarray}
In particular~\eqref{eq:t large} implies that 
$$ 
u_{0,m} + h_m(g_m-1)^{-1} \not \equiv 0 \pmod p
$$ 
as otherwise
$$
u_{1,m} \equiv   g_m u_{0,m} + h_m \equiv  u_{0,m}   \pmod p
$$
and $t=1$.

 We now recall that by the result of~\cite{BGK}, 
 for any $\varepsilon> 0$ there exists $\eta >0$
 such that under the condition~\eqref{eq:t large} we have
$$
\sum_{n=1}^t \ep(c g_m^n) \ll  t p^{-\eta}
$$ 
which concludes the proof in the case of $g_m > 1$. 

For   $g_m= 1$ we recall~\eqref{eq:um expl = 1} and then 
using~\eqref{eq:lin sum} we derive the result.   
\end{proof}

Using again Lemma~\ref{lem:Kok-Szu},
we  derive the following generalisation 
of~\cite[Theorem~6]{OstShp} (the bound is $\log p$ weaker
as we work in the dimension $m+1$ instead of $m$).

\begin{cor}
\label{cor:Discr m} 
Let  the sequence $\{\vec{u}_n\}$ be given by~\eqref{eq:Gen}
for     $\cF_k = \cF$, $k=1,2, \ldots$, with a 
polynomial 
system $\cF \in  \fF(S,m)$ of the form~\eqref{eq:Polys}
of total degree $d \ge 2$  
and such that $s_{0,1} \ldots s_{m-1,m} \ne 0$.
Assume that $\{\vec{w}_n\}$ is purely
periodic with period $\tau$.
 Then for any fixed real $\varepsilon > 0$,
there exist $\gamma > 0$ such that 
for any positive integer  $N$ with 
$\tau \ge N \ge p^{m+\varepsilon}$ the discrepancy   of the sequence 
$$
\(\frac{u_{n,0}}{p}, \ldots, \frac{u_{n,m}}{p}\),
\qquad n = 0,\ldots, N-1,
$$
satisfies the bound  
$O\(p^{-\gamma}\)$,
where the implied constant depends  only on $d$, $m$
and $\varepsilon$.
\end{cor}

Certainly one can get stronger and more explicit statements 
in both  Theorem~\ref{thm:ExpSum R}
and Corollary~\ref{cor:Discr m} if more information 
about the multiplicative order $t$ modulo $p$ is available.
For example, if it is know that $t \ge p^{1/3 + \varepsilon}$
then one can use  the bound of 
Heath-Brown and  Konyagin~\cite{HBK}  
(see also~\cite[Theorem~3.4]{KoSh})
$$
\sum_{n=1}^t \ep(c g_m^n) \ll 
\min\{  p^{1/2}, p^{1/4}t^{3/8}, p^{1/8}t^{5/8}\}.
$$
For smaller values of $t$, but with 
$t \ge p^{1/4}$ one can use the bound
of Bourgain and Garaev~\cite{BouGar}, see also~\cite{Kon}. 

We remark that it is easy to see that a randomly chosen element 
$g \in \F_p^*$  is of order $t = p^{1+o(1)}$ with probability
$1 + o(1)$ as $p\to \infty$. 

Furthermore, it is also well-known that any fixed integer 
$g \ne 0, \pm 1$  is of multiplicative order
\begin{equation}
\label{eq:t sqrt}
t \ge p^{1/2},
\end{equation}
for all but $o(x/\log x)$ primes $p\le x$, see~\cite{ErdMur,IndlTim,Papp}
for various improvements of this result.

\subsection{Permutation Systems}
\label{eq:PerSys}

We now consider polynomial systems of the form~\eqref{eq:Poly Syst}  which  
permute the elements of $\Fp^{m+1}$. 
Lidl and Niederreiter~\cite{LN,LN1} call such  systems 
{\it orthogonal polynomial systems\/}, but we here refer to
them as {\it permutation polynomial systems\/}.

We fix a sequence $\cF_k$, $k =1,2, \ldots$, of polynomial 
systems~\eqref{eq:Poly Syst}. 
For integer vectors $\vec{b} = (b_0, \ldots, b_{m-1}) \in \Fp^m$
and $\vec{a} = (a_0, \ldots, a_{m}) \in \Fp^{m+1}$ and 
integers $c,M,N$ with $M \ge 1$ and $N \ge 1$, we consider the average
values of  exponential sums
\begin{equation*}
\begin{split}
U_{\vec{a},c}(M,N) &=  \sum_{w_{0},\ldots, w_{m}\in \F_{p}}
\left|\sum_{n=0}^{N-1}
\ep \(\sum_{j=0}^{m-1} a_j F_j^{(n)}(w_0,\ldots,w_m)\) \eM(c n)\right|^2,\\
V_{\vec{b},c}(M,N) &=  \sum_{w_{0},\ldots, w_{m}\in \F_{p}}
\left|\sum_{n=0}^{N-1}
\ep \(\sum_{j=0}^{m} b_j F_j^{(n)}(w_0,\ldots,w_m)\) \eM(c n)\right|^2,
\end{split}
\end{equation*}
where, as before, the polynomials $F_i^{(k)}$, $i=0,\ldots,m$, $k =1,2, \ldots$ are
given by~\eqref{eq:Poly Rec}. 

Then using Lemma~\ref{lem:LinTerm+Deg}  in the argument 
of~\cite{Ost} one immediately obtains the following generalisation of the
bound of exponential sums from~\cite{Ost}.

\begin{theorem}
\label{thm:ExpSumAvV} Assume that $\cF_k \in \fF(S,m)$, $k =1,2, \ldots$, are
permutation polynomial  systems~\eqref{eq:Poly Syst}, 
and such that $s_{0,1} \ldots s_{m-1,m} \ne 0$.
Then for any  positive integers $c$, $M$, $N$ 
and any nonzero vector $\vec{b}  \in \F_p^m$
we have
$$
U_{\vec{a},c}(M,N)  \ll A(N,p),
$$
where
$$
A(N,p) =
\left\{ \begin{array}{ll}
N p^{m+1}  & \mbox{if}\ N \le p^{1/{(m+1)}}, \\
N^2 p^{m(m+2)/(m+1)} & \mbox{if}\ N > p^{1/{(m+1)}}.
\end{array} \right.
$$
\end{theorem}

Exactly as in~\cite{Ost}, this immediately implies a discrepancy bound which 
holds for almost all initial values $\vec{v} \in \F_p^{m+1}$. 
 We note that 
in~\cite{Ost} only the case of when at each step 
the same polynomial system $\cF_k = \cF$ is applied but the 
proof, based only on the bound of the sums $U_{\vec{a},c}(M,N)$,  holds for distinct polynomial systems $\cF_k \in \fF(S,m)$ 
without any changes.

\begin{cor}
\label{cor:DiscrAvu}
Let $0<\varepsilon<1$ and let the sequence $\{\vec{u}_n(\vec{v})\}$ be given by~\eqref{eq:Gen}
with the initial vector of initial values $\vec{v} \in \F_p^{m+1}$, where
$\cF_k \in \fF(S,m)$, $k =1,2, \ldots$, are
permutation polynomial  systems~\eqref{eq:Poly Syst}, 
and such that $s_{0,1} \ldots s_{m-1,m} \ne 0$. Then for all initial values $\vec{v}\in \Fp^{m+1}$  except at most
$O(\varepsilon p^{m+1})$, 
and any positive integer $N \le p^{m+1}$, 
the discrepancy $D_N(\vec{v})$ of the sequence
$$
\(\frac{u_{n,0}(\vec{v})}{p}, \ldots, \frac{u_{n,m-1}(\vec{v})}{p}\),
\qquad n = 0,\ldots, N-1,
$$
satisfies the bound 
$$
D_N(\vec{v})  \ll \varepsilon^{-1} C(N,p),
$$
where
$$
C(N,p) =
\left\{ \begin{array}{ll}
N^{-1/2}(\log N)^{m+1} \log p & \mbox{if}\ N \le p^{1/{(m+1)}}, \\
p^{-1/2(m+1)} (\log N)^{m+1} \log p & \mbox{if}\ N > p^{1/{(m+1)}}.
\end{array} \right.
$$
\end{cor}

We now show that the distribution of the full 
vectors $\{\vec{w}_n(\vec{v})\}$
can be studied as well.

\begin{theorem}
\label{thm:ExpSumAvU} Let $\cF_k \in \fF(S,m)$ be a sequence of 
permutation polynomial  systems~\eqref{eq:Poly Syst} 
and such that $s_{0,1} \ldots s_{m-1,m} \ne 0$, 
satisfying also the additional condition that the last polynomial in all these systems has the same coefficient $g_m\in\Fp$ of $X_m$, that is,
$$F_{k,m}(X_0, \ldots ,X_m) = g_mX_m+h_{k,m}, \qquad k =1,2, \ldots.
$$ 
Denote by $t$ the period of $g_m$ if $g_m  \ne 1$ and put $t=p$ if $g_m = 1$.
Then for any  positive integers $c$, $M$,  $N$ and any nonzero 
vector $\vec{b} \in \F_p^{m+1}$
we have
$$
V_{\vec{b},c}(M,N)  \ll B(N,t,p), 
$$
where 
$$
B(N,t,p)=A(N,p)+N^2 t^{-1} p^{m+1} 
$$
and $A(N,p)$ is defined as in Theorem~\ref{thm:ExpSumAvV}. 
\end{theorem}

\begin{proof}
 Note, as before, that if 
$\gcd(b_0, \ldots, b_{m-1} ,p) = 1$ then the proof 
of~\cite[Lemma~4]{Ost} applies to the sums  
$V_{\vec{b},c}(M,N)$ without any changes. So it remains to consider 
the case 
$$
b_0 \equiv \ldots  \equiv b_{m-1}  \equiv 0 \pmod p \mand 
\gcd(b_{m} ,p) = 1, 
$$
in which case we simply obtain 
\begin{eqnarray*}
\lefteqn{
V_{\vec{b},c}(M,N) =  \sum_{v_{0},\ldots, v_{m}\in \F_{p}}
\left|\sum_{n=0}^{N-1}
\ep \(b_m F_m^{(n)}(v_0,\ldots,v_m)\) \eM(c n)\right|^2}\\
&& \quad =  \sum_{k,n=0}^{N-1} \eM(c (k-n))\\
& & \qquad \qquad 
\sum_{{v_{0},\ldots, v_{m}\in \F_{p}}}\ep \(b_m\(F_m^{{(k)}}(v_0,\ldots,v_m) -
F_m^{{(n)}}(v_0,\ldots,v_m)\)\) \\
&& \quad\le \sum_{k,n=0}^{N-1} \left| \sum_{{v_{0},\ldots, v_{m}\in \F_{p}}} \ep \( b_m \(F_m^{{(k)}}(v_0,\ldots,v_m) -F_m^{{(n)}}(v_0,\ldots,v_m )\)\) \right|.
\end{eqnarray*}

We have the follwing explicit formulas (see also~\eqref{eq:um expl ne 1} 
and~\eqref{eq:um expl = 1}): 
$$
F_m^{(k)} = g_m^k X_m + d_m \qquad k =0, 1, \ldots, 
$$
if $g_m\ne 1$  and 
\begin{equation}
\label{eq:fm expl = 1}
F_m^{(k)} = X_m+\sum_{i=1}^k h_{i,m}, \qquad k =0, 1, \ldots, 
\end{equation}
if $g_m = 1$,
where 
$$
d_m = \sum_{i=1}^k g_m^{k-i} h_{i,m},
$$

We treat first the case $g_m\ne 1$. In this case we get:
\begin{eqnarray*}
\lefteqn{V_{\vec{b},c}(M,N) \le  \sum_{k,n=0}^{N-1} \left| \sum_{{v_{0},\ldots, v_{m}\in \F_{p}}} \ep \( b_m \((g_m^k-g_m^n) v_m+d_k-d_n\)\) \right| }\\
& & = \sum_{\substack{k,n=0\\k \equiv n\pmod t}}^{N-1} \left| \sum_{{v_{0},\ldots, v_{m}\in \F_{p}}} \ep \( b_m \((g_m^k-g_m^n) v_m+d_k-d_n\)\) \right| \\
& &   +~\sum_{\substack{k,n=0\\k\not \equiv n\pmod t}}^{N-1} \left| \sum_{{v_{0},\ldots, v_{m}\in \F_{p}}} 
\ep \( b_m \((g_m^k-g_m^n) v_m+d_k-d_n\)\) \right|.
\end{eqnarray*}
Because $g_m^k-g_m^n\equiv 0 \pmod p$ if and only if $k\equiv n \pmod t$, we estimate the first sum trivially as $N(N t^{-1}+1) p^{m+1}$. Furthermore, 
for $k\not \equiv  n \pmod t$, using~\eqref{eq:lin sum}  we see 
that the second sum simply vanishes.

Thus, for $g_m\ne 1$, we obtain
$$
V_{\vec{b},c}(M,N) \ll A(N,p)+N(N t^{-1}+1) p^{m+1} = A(N,p)+ N^2 t^{-1} p^{m+1}.
$$

For the case $g_m=1$ we recall~\eqref{eq:fm expl = 1} and using 
similar arguments easily derive the desired result.  
\end{proof}

As above, we now get:

\begin{cor}
\label{cor:DiscrAvv}
Let $0<\varepsilon<1$ and let the sequence $\{\vec{u}_n\}$ be given by~\eqref{eq:Gen}, where
$\cF_k \in \fF(S,m)$ is a sequence of 
permutation polynomial  systems~\eqref{eq:Poly Syst} 
satisfying also the additional condition that the last polynomial in all these systems has the same coefficient $g_m\in\Fp$ of $X_m$, that is,
$$F_{k,m}(X_0, \ldots ,X_m) = g_mX_m+h_{k,m}, \qquad k =1,2, \ldots.
$$  
Denote by $t$ the period of $g_m$ if $g_m \ne 1$ and put $t=p$ if $g_m= 1$.  
Then for all vectors of initial values $\vec{v}\in \Fp^{m+1}$  except at most
$O(\varepsilon p^{m+1})$, 
and any positive integer $N \le p^{m+1}$, 
the discrepancy $D_N(\vec{v})$ of the sequence
$$
\(\frac{u_{n,0}(\vec{v})}{p}, \ldots, \frac{u_{n,m}(\vec{v})}{p}\),
\qquad n = 0,\ldots, N-1,
$$
satisfies the bound 
$$
D_N(\vec{v})  \ll \varepsilon^{-1} D(N,t,p),
$$
where
$$
D(N,t,p) = C(N,p) \log N +  t^{-1/2} (\log N)^{m+2} \log p
$$
and $C(N,p)$ is defined as in Corollary~\ref{cor:DiscrAvu}. 
\end{cor}

It is easy to see that under the condition~\eqref{eq:t sqrt} 
the quantities  $B(N,t,p)$ and $D(N,t,p)$ are dominated 
by the terms with $A(N,p)$ and $C(N,p)$, respectively:
$$
B(N,t,p)\ll A(N,p) \mand D(N,t,p) \ll C(N,p) \log N.
$$

Finally, we remark that analogues of Theorem~\ref{thm:ExpSumAvU}
and  Corollary~\ref{cor:DiscrAvv} 
can be proven also for more general permutation polynomial systems, namely for systems in which the coefficients $g_{j,m}$ of $X_m$ in the last polynomial of each system vary
in such a way  that  
\begin{equation}
\label{eq:products}
\prod_{j = 1}^k g_{j,m} \not \equiv \prod_{j = 1}^n g_{j,m} \pmod p
\end{equation}
is  $k$ and $n$ are close to each. In fact, if this is guaranteed for 
$k$ and $n$ with $0 < |k-n|<t$ then the corresponding results for such 
polynomial systems look identical to those of Theorem~\ref{thm:ExpSumAvU}
and  Corollary~\ref{cor:DiscrAvv}. 
For examples included such  sequences of coefficient 
as  $g_{j,m} = g_m^j$ for some element $g_m\in\Fp^*$. In this case,
the condition~\eqref{eq:products} is equivalent to the quadratic congruence
$$
k(k+1) \equiv n(n+1) \pmod {2t},
$$
where $t$ is the order of $g_m$
which can be easily shown not to have too many solutions with
$0 \le k,n \le N-1$ (in particular, if $t$ is prime the results are 
again exactly the same as those of Theorem~\ref{thm:ExpSumAvU}
and  Corollary~\ref{cor:DiscrAvv}). 

\section{Hash Functions from Polynomial Iterations}

\subsection{General Construction} In this section we propose a new construction of hash functions based on iterations of polynomial systems studied in the previous sections. This construction is motivated 
by that of D.~X.~Charles, E.~Z.~Goren and
K.~E.~Lauter~\cite{CGL} and in some sense it may be considered as 
its extension. 

Let $n$ and $r$ be two nonzero integers. Choose a random $n$-bit prime $p$ and $2^r$ permutation polynomial systems $\cF_\ell$, $\ell = 0, \ldots, 2^r-1$,
not necessary distinct,
defined by~\eqref{eq:Poly Syst} and~\eqref{eq:Poly Rec}.

We also consider a random initial vector $\vec{w}_0 \in \F_p^{m+1}$.

As in~\cite{CGL}, the input of the hash function is used to decide what polynomial system $\cF_\ell$ is used to iterate. More precisely,
it works as follows given an   input bit string $\Sigma$, we
execute the following steps:

\begin{itemize}
\item   pad $\Sigma$ with at most $r-1$ zeros on the left to make
sure that its length $L$ is a multiple of $r$;

\item split $\Sigma$   into blocks $\sigma_j$, $j =1, \ldots,J$, 
where $J = L/r$, of length $r$ and interpret each block
as an integer $\ell \in [0, 2^r-1]$.

\item Starting at the vector $\vec{w}_0$, apply the polynomial systems
$\cF_\ell$ iteratively obtaining the sequence of vectors $\vec{w}_j \in \F_p^{m+1}$.

\item Output $\vec{w}_J$ as the value of the hash function
(which can also be now interpreted as  a binary $(m+1)n$-bit string).
\end{itemize}

The above construction is quite similar to that of~\cite{CGL}
where $m=1$, the vectors $\vec{w}_j$ represent
 the coefficients of an equation describing an elliptic curve
for example, of the Weierstrass equation
$$
Y^2 = X^3 +sX + r
$$
and polynomials maps are associated with isogenies of
a fixed degree.

\subsection{Collision Resistance}

Our belief in collision resistance is essentially based on 
the same arguments as in~\cite{CGL}.

We remark  that  the initial vector $\vec{w}_0$ 
is fixed and in particular, does not depend on the input of the hash function. Furthermore, the collision resistance does not 
rely on the difficulty of 
inverting the maps generated by the polynomial systems $\cF_\ell$,
which are triangular and actually quite easy  to invert. 
Rather, it is based on the difficulty of making the decision which
system to apply at each step when one attempts to back trace 
from a given output to the initial vector  $\vec{w}_0$ and thus
produce two distinct strings $\Sigma_1$ and
$\Sigma_2$ of the same length $L$,  with the same output. 

Note that for strings of different lengths, say of $L$ and $L+1$,  a collision can easily be created. It is enough to take 
$\Sigma_2 = (0, \Sigma_1)$ (that is, $\Sigma_2$ is obtained
from $\Sigma_1$ by augmenting it by $0$). If $L \not \equiv 0 \pmod r$
then they lead to the same output. 
Certainly any  practical implementation has to take care of 
things like this.

We also note that the results of Section~\ref{eq:PerSys}
suggest that the above hash functions exhibit rather 
chaotic behaviour, which is close to the behaviour of
a random function.
We certainly make no claims about the 
cryptographic strength of our construction but we believe that 
there are enough reasons to investigate it (theoretically 
and experimentally)  more closely. 

\section{Remarks}

In the proof of Lemma~\ref{lem:Elem} we use
the estimate $O\(\deg \Phi_{s,\vec{k},\vec{l}} p^{m-s-1}\)$
on the number of zeros of  
the polynomial  $\Phi_{s,\vec{k},\vec{l}}$.
Perhaps this bound is hard to improve in general, 
but maybe this can be done for some specially 
selected polynomial systems. For example, if one can show
that  $\Phi_{s,\vec{k},\vec{l}}$ is absolutely 
irreducible then the Lang-Weil bound on the number of zeros of a polynomial in $m\ge2$ variables, see~\cite{LaWe,Sch},
can be used to derive a better result. Even the case of 
$\nu =1$ is already of interest. 

Furthermore, although low discrepancy is a very important 
requirement on any  pseudoramdom number generator, this is
not the only one. For example, the notion of 
{\it linear complexity\/} 
also plays an important role in this area, see~\cite{TopWin}. 
In the case of vector sequences it is natural to
consider linear relations with vector 
coefficients. 
Namely, we denote by $L(N)$ the smallest $L$
such that for some  $m$-dimensional 
vectors $\vec{c}_0, \ldots, \vec{c}_L$ 
over $\F_q$ where $\vec{c}_L$ is a non-zero 
vector,   we have
\begin{equation}
\label{eq:lincomp}
\sum_{h=0}^L \vec{c}_h \cdot \vec{u}_{n+h} = 0
\end{equation}
for all $h =0, \ldots, N-L-1$, where 
$\vec{c} \cdot \vec{u}$ denotes the scalar product.  
Using the same degree argument which is used in the 
proof of  Lemma~\ref{lem:Elem},  
we see that~\eqref{eq:lincomp} 
leads to a nontrivial polynomial equation 
in $m+1$ variables over $\F_p$ of degree $O(L^m)$.
Since for $N \le \tau$,  where as $\tau$ is the period of
the purely periodic sequence $\{\vec{w}_n\}$, the vectors 
$\vec{w}_{n+h}$, $h =0, \ldots, N-L-1$ are pairwise distinct,
this yields the estimate
$$
L(N) \gg N^{1/m} p^{-1}, \qquad 0 \le N \le \tau.
$$
This can be extended to sequences over arbitrary finite fields. 
Several more estimates of this type have recently been
given in~\cite{OstShpWing}. 
It would be very interesting to get 
better bounds which rely on a 
more refined analysis of~\eqref{eq:lincomp}. 
 
\section*{Acknowledgement}

The authors are grateful to the Fields Institute 
for its support and stimulating atmosphere which 
led to the initiation of this work at 
the ``Fields Cryptography Retrospective Meeting''
Toronto, 2009. 

During the preparation of this paper,  
A.~O. was supported in part by 
the Swiss National Science Foundation   Grant~121874  
and I.~S. by
the  Australian Research Council 
Grant~DP0556431.

\end{document}